\documentclass{amsart}
\usepackage{amsmath,amssymb,amsthm,amscd,graphicx,color,accents}

\usepackage[colorlinks=true,linkcolor=blue,citecolor=blue,urlcolor=blue]{hyperref}

\newtheorem{thm}{Theorem}[section]
\newtheorem{prop}[thm]{Proposition}
\newtheorem{lem}[thm]{Lemma}

\theoremstyle{definition}

\theoremstyle{remark}

\numberwithin{equation}{section}

\newcommand{\R}{\mathbb{R}}  
\newcommand{\C}{\mathbb{C}}  

\date{\today}

\begin{document}

\title{GENERALIZED WATSON TRANSFORMS II: THE COMPLEMENTARY SERIES OF $\boldsymbol{GL(2,\R)}$}

\author{Qifu Zheng}
\address{Department of Mathematics and Statistics, The College of New Jersey, Ewing, NJ 08618}
\email{zheng@tcnj.edu}

\begin{abstract}
We apply the theory of generalized Watson transforms developed in \cite{zheng00} to construct the complementary series of $GL(2,\R)$.
\end{abstract}

\textbf{\maketitle}

\parindent=25pt
\pagestyle{myheadings}

\markboth{\hfill\sc{Q. Zheng}\hfill}{\hfill\sc{Generalized Watson Transforms, II: The Complementary Series of }$GL(2,\R)$\hfill}

\section{Introduction}
\label{Intro}

This article is the second in a series of articles in which we develop the theory of generalized Watson transforms and make applications of those results to the representation theory of the general linear groups over $\R$.  It is well known \cite{bargmann,knapp} that the irreducible representations of $GL(2,\R)$, the general linear group of $2 \times 2$ real matrices, are classified according to three distinct constructions: (1) The principal series are usually constructed by unitary induction from its parabolic subgroup. (2) The complementary series are constructed by a form of analytic continuation from the principal series. (3) The (relative) discrete series are usually constructed in spaces of holomorphic functions on the unit disk or upper half complex plane.  

On the other hand, by applying the results developed in [3], we can obtain all three series using the method of generalized Watson transforms. That this method is able to achieve these results is due to the fact that the group $G = GL(2,\R)$ is generated by its upper triangular (Borel) subgroup $Q$ and the Weyl reflection, 
\begin{equation}
\label{pmatrix}
p = \begin{bmatrix}
 0 & 1 \\
 -1 & 0  \\
 \end{bmatrix}
 \end{equation}
then, any irreducible unitary representation $\pi$ of $G$ is determined by its restriction to $Q$ and $p$, which in fact corresponding to a generalized Watson transform.  In this paper, we will illuminate this approach by applying the method of generalized Watson transforms to construct the complementary series of $G$, and in a subsequent article \cite{zheng_in_prep}, we will use the generalized Watson transform method to construct unitary representations of higher rank groups.
 
 This paper is organized as follows: In Section \ref{Review}, we review briefly some concepts and
 theorems related to generalized Watson transforms from \cite{zheng00}.  In Section \ref{pitt}, we will describe the subgroups of $G$ and its non-unitary representations realized on the Hilbert space $L^2(\R, (1+ x^2)^s dx)$.  Then, in Section \ref{Subgroup} we will use Pitt's theorem \cite{pitt} to realize the representations on the space $H_s = L^2(\R, |x|^{-s}dx)$ where $0 < s < 1$.  Finally, in Section \ref{Proof} we will show that the representations realized on $H_s$ in Section 4 are unitary.

 \section{Some remarks on the generalized Watson transforms}
 \label{Review}
 

Let $G_0$ be a topological group, $R$ and $L$ be unitary representations of $G_0$ on a Hilbert space $H$, and let $I$ denote the identity operator on $H$.  A unitary operator $W$ that intertwines $R$ and $L$ is called a {\it generalized Watson transform with respect to $R$ and $L$} if $W^2 = I$.  The operator $W$ is called a {\it generalized skew Watson transform with respect to $R$ and $L$} if $W^2 = -I$.   The results in \cite{zheng00} provide several theorems on the construction of generalized Watson transforms.  Here, we list one corollary that is needed in the proof of the unitarity of the complementary series.

 \begin{prop} {\rm(Zheng \cite{zheng00})}
 \label{RandL}
 Suppose that $G_0$ is Abelian, let $R$ be a unitary representation of $G_0$ on a Hilbert space $H$, and set $L( g ) = R(g^{-1})$ for all $g \in G_0$.  For $\phi\in H$, suppose that $\phi^{\circ}=\left\{R(g)\phi : g\in G_0\right\}$ spans a dense subspace of $H$. Then there exists a generalized Watson transform $W$ on $H$ with respect to $R$ and $L$ such that $W\phi = \pm\phi$ if and only if $\langle\phi|R(g)\phi\rangle $ is real for all $g \in G_0$. 
 \end{prop}

 \section {Subgroups of \texorpdfstring{$G$}{G} and its non-unitary representations}
 \label{Subgroup}
 
Denote the elements of $G$ by
$$
 g=\begin{bmatrix}
  a & c \\
  b & d  \\
  \end{bmatrix},
$$ 
 where $a, b, c, d \in\R$ and $ad - bc \ne 0$.  Let $Q$ be the full upper-triangular subgroup of matrices
\begin{equation}
 \label{qmatrix}
 q := q(a,c,d)=\begin{bmatrix}
   a & c \\
   0 & d  \\
   \end{bmatrix},
\end{equation}
 and $Q^t$ be the analogous full lower-triangular subgroup. Then $Q$ is the semi-direct product of the normal Abelian subgroup $N$ of unipotent matrices
\begin{equation}
 \label{nmatrix}
 n := n( c) = \begin{bmatrix}
   1 & c \\
   0 & 1  \\
   \end{bmatrix},
\end{equation}
$c \in \R$, and the diagonal subgroup $D$ of matrices
\begin{equation}
 \label{gammamatrix}
 \gamma(a, d)=
 \begin{bmatrix}
    a & 0 \\
    0 & d  \\
 \end{bmatrix},
\end{equation}
 with $a \ne 0, d\ne 0$. The matrix $p$ in (\ref{pmatrix}), 
 called the Weyl element, plays a special role in the representation theory of $G$.  Since $p^2 = -I$, the generalized Watson transforms are operators associated by representations to $p$. This explains the importance of generalized Watson transforms in the representation theory of $G$, and more generally, of reductive Lie groups.
 
 The following result is well known \cite{knapp}.  

 \begin{prop}
 \label{NandD}
 The subgroups $N$ and $D$ and the Weyl element $p$ generate the group $G$.
 \end{prop} 
 
The proof of the above result rests on two identities. First, when $b = 0$,
$$g=\begin{bmatrix}
   a & c \\
   0 & d  \\
   \end{bmatrix}=q(a,c,d)=\gamma(a,d) \, n(\frac{c}{a});$$
   and when $b\ne 0$,
 $$g=\begin{bmatrix}
    a & c \\
    b & d  \\
    \end{bmatrix}=n(a/b) \, p \, \gamma(-b,-(\det{g})/b) \, n(d/b).
    $$
Therefore, any representation of the group $G$ is completely determined by its restrictions to $N,D$ and $p$. 
 
For $d \in \R$, let $\sigma(d)$ denote the sign of $d$; also, let $\epsilon$ equal 0 or 1.  Starting from the one-dimensional character of $Q^t$, 
 $$
 \tau_s\left(\begin{bmatrix}
    a & 0 \\
    b & d  \\
    \end{bmatrix}\right)=\sigma(d)^{\epsilon}|a|^{-(\Re(s)+1)/2}|d|^{(\Re(s)+1)/2}$$
we can obtain the bounded representation $\pi_s$, of $G$ on the weighted norm space $V_s= L^2 (\R,(1+x^2)^{\Re(s)}dx)$, defined by
 \begin{multline}
 \label{pis}
(\pi_s(g)f)(x) = (\sigma(\det g))^{\epsilon}(\sigma(\det(a +xb)))^{\epsilon} \\
\times |\det g|^{(\Re(s)+1)/ 2} |\det(a + xb)|^{ -( \Re(s)+1 )}f\left(\frac{c+xd}{a+xb}\right),
 \end{multline}
The representation $\pi_s$ on $V_s$ is unitary if and only if $s$ is pure imaginary, i.e., $\Re(s) = 0$.  Since in this article we concentrate on the complementary series, we will only consider the case in which $\epsilon = 0$ and $0 < s < 1$, and in that case, (\ref{pis}) becomes
\begin{equation}
\label{special}
  (\pi_s(g)f)(x) = |\det g|^{(s+1)/ 2} |\det(a + xb)|^{ -(s+1)}f\left(\frac{c+xd}{a+xb}\right)
\end{equation}
 
 \section{New realizations of non-unitary representations of \texorpdfstring{$G$}{G}}
 \label{pitt}
 
 In this section, we will make use of a well-known theorem of Pitt \cite{pitt} to realize the non-unitary representation  $\pi_s$ on the weighted norm space $H_s = L^2(\R, |x|^sdx)$ when $0 < s < 1$.  We first state a special case of Pitt's theorem and refer readers to Stein \cite{stein} for more general versions of Pitt's theorem.
 \begin{thm} 
 \label{pitthm}
 {\rm (Pitt \cite{pitt})} Let $s\in (0,1)$ and let $\hat{f}$ denote the Fourier transform of a function $f:\R\to\C$.  Then there exists a constant $C$ such that
 \begin{equation}
 \label{inequ}
 \int_{-\infty}^{\infty}|\hat{f}(x)|^2|x|^{-s}dx \le C\int_{-\infty}^{\infty}|\hat{f}(x)|^2|x|^{s}dx
 \end{equation}
 for every function $f \in H_s$ for which the integral on the right side of (\ref{inequ}) is convergent.
 \end{thm}

By applying Proposition \ref{pitthm}, we can define in terms of $\pi_s$ a representation $\rho_s$ on $H_s$.  Indeed, denoting by $\mathcal{F}$ the Fourier transform, we have the following result.

 \begin{lem}
Define, for $g \in G$, the operator 
\begin{equation}
\label{rhos}
\rho_s(g) = \mathcal{F}\pi_s(g){\mathcal{F}}^{-1}.
\end{equation}
Then $\rho_s$ is a well-defined representation of $G$ on $H_s$, $0 < s < 1$.
 \end{lem}
 
 \begin{proof}
 By Pitt's theorem, there exists a constant $C > 0$ such that for any $f \in V_s$
\begin{align*}
\int_{-\infty}^{\infty}|(\mathcal{F}{f})(x)|^2 \, |x|^{-s}dx &\le C\int_{-\infty}^{\infty}|f(x)|^2 \, |x|^{s}dx \\
&\le C\int_{-\infty}^{\infty}|f(x)|^2|1+x^2|^{s}dx,
\end{align*}
$0 < s < 1$.  Hence, the Fourier transform $\mathcal{F}$ is a continuous map from $V_s$ to $H_s$, and this shows that $\mathcal{F}V_s \subset H_s$.

Define the functions $\lambda, \mu:\R\to\R$ by  $\lambda(x) = e^{-x^2/2}$ and $ \mu(x) = xe^{-x^2/2}$.  
 Then it is a simple calculation to show that both $\lambda$ and $\mu$ are in $V_s$ and $H_s$, and also that $\mathcal{F}\lambda=\mu,\mathcal{F}\mu=-i\mu$.   Define a homomorphism $T$ of $\R^{\times}$, the multiplicative group of non-zero real numbers, by 
 $$
 (T(a)f)( x) = a^{1/2}f(ax)
 $$
 for any $a \in \R^{\times} $ and any function $f$ on $\R$.  Then for $f \in V_s$, we obtain the relation, 
 \begin{equation}
 \mathcal{F}T(a)=T(a^{-1})\mathcal{F}.
 \end{equation}
 By the uniqueness property of Laplace transform, it follows that if $f \in V_s$ is such that for any $a \in \R^{\times} $
 $$
 \int_{-\infty}^{\infty}f(x)(T(a)\lambda)(x)(1+x^2)^{s}dx=0
 $$
 and 
  $$
  \int_{-\infty}^{\infty}f(x)(T(a)\mu)(x)(1+x^2)^{s}dx=0
  $$
 then $f = 0$. Hence, $\hbox{Span}\!\left\{T(a)\lambda, T(a)\mu |a \in \R^{\times}\right\}$ is a dense subspace of $V_s$.  Similarly, the space generated by its image, {\it viz.}, $\hbox{Span}\!\left\{T(a^{-1})\lambda, T(a^{-1})\mu |a \in \R^{\times}\right\}$, is dense in $H_s$.  Therefore, it follows from the continuity of $\mathcal{F}$ that $\mathcal{F}V_s = H_s$.  
 
 Consequently, for any $f \in H_s$ and $g \in G$, we obtain 
${\mathcal{F}}^{-1}f \in V_s,$ $\pi_s (g){\mathcal{F}}^{-1}f\in V_s,$ and
$ \mathcal{F}\pi_s(g){\mathcal{F}}^{-1}f \in H_s.$
 Therefore, we have proved that $\rho_s(g ) = \mathcal{F}\pi_s(g){\mathcal{F}}^{-1}$  is well-defined for any $g \in G$.
 Finally, since $\pi_s$ is a representation of $G$ on $V_s$ then it follows immediately that $\rho_s$ also is a well-defined
 representation of $G$ on $H_s.$ 
 \end{proof}

\section {Unitarity of the complementary series}
\label{Proof}

 Throughout this section, we will use the notation defined in (\ref{qmatrix})-(\ref{gammamatrix}) for the  elements of the subgroups of $G$.  The main of this section is to establish the following result.

 \begin{thm}
 \label{main}
 For $s \in (0,1)$, the operator $\rho_s$ defined in ($\ref{rhos}$) is a unitary representation of $G$ on $H_s$.
  \end{thm}
  
Before embarking on the proof of Theorem \ref{main}, we shall establish several preliminary results.

  \begin{lem}
  \label{preserve}
  The subgroup $Q$ of upper-triangular subgroup preserves the norm of $H_s$ when $\rho_s$ is restricted to $Q$.
 \end{lem}
 
 \begin{proof}
  Notice that if $q = q( a, c, d ) \in Q$ then, by $(\ref{special})$, 
  \begin{equation}
 (\pi_s(q)f)(x) = |a|^{-(s+1)/2}|d|^{(s+1)/2}f(a^{-1}(c+xd )).
  \end{equation}
  Applying the Fourier transform, we obtain 
 \begin{align*}
 (\rho_s(q)f)(x) &= (\mathcal{F}\pi_s(q)\mathcal{F}^{-1}f)(x) \\
 &= e^{icd^{-1}x}|a|^{(1-s)/2}|d|^{(s-1)/2}f(d^{-1}xa).
 \end{align*}
 Hence the result follows by a simple calculation.
 \end{proof}
 
Define
 $
 R(\gamma)=\rho_s(\gamma)
 $ 
 for $\gamma= \gamma(a, d) \in D$. Then, it follows from the above lemma that $R$ is a unitary representation of $D$ on $H_s$ and
 $$
 (R(\gamma)f)(x) = |a|^{(1-s)/2}|d|^{(s-1)/2}f(d^{-1}xa).
 $$
 
\begin{lem}
\label{phi}
 For $s \in \C$ such that $0< \Re(s) < 1$, define 
 \begin{equation}
 \phi_s(x)=\int_{-\infty}^{\infty} e^{-ixy} \, (1+y^2)^{-(s+1)/2} \, dy.
 \end{equation}
 Then $\phi_s \in H_s$, and the set ${\mathbf{\Phi}}^{\circ} = \{R(\gamma)\phi_s:\gamma \in D\}$ spans a dense subspace of $H_s^+ =\{f \in H_s: f \text{ is even}\}$.
\end{lem}

\begin{proof}
Since $0 < \Re(s)< 1$ then the functions $(1+y^2)^{-(s+1)/2}$ and $(1+y^2)^{-(s+1)}$ are elements of $L^1(\R)$.  Therefore, by the $L^1$ and $L^2$ properties of the Fourier transform, it follows that $\phi_s \in L^2 (\R)\cap L^{\infty}(\R)$, and
 \begin{align*}
\langle\phi_s|\phi_s\rangle_s &= \int_{-\infty}^{\infty}|\phi_s(x)|^2|x|^{-s}dx \\
 &= \int_{|x|\le 1}|\phi_s(x)|^2|x|^{-s}dx +\int_{|x|>1}|\phi_s(x)|^2|x|^{-s}dx \\
 &\le \|\phi_s\|_{\infty} \int_{|x|\le 1}|x|^{-s}dx+\int_{|x|>1}|\phi_s(x)|^2dx\\
 &\le 2(1-s)^{-1} \|\phi_s\|_{\infty} + \|\phi_s\|_2^2.
 \end{align*}
 where $\|\phi_s\|_{\infty}$ and $\|\phi_s\|_2$ are the norms of $\phi_s$ in $L^{\infty} (\R )$ and $L^2 (\R)$, respectively.
 Therefore $\phi_s \in H_s$, and
 \begin{align*}
 \phi_s(x) &= \int_{-\infty}^{\infty} e^{-ixy} \, (1+y^2)^{-(s+1)/2} \, dy \nonumber \\
 &= \frac{1}{\Gamma((1+s)/2)} \int_{-\infty}^{\infty} e^{-ixy} \left[\int_{0}^{\infty} \xi^{(s-1)/2} e^{-\xi(1+y^2)}d\xi\right]dy.
 \end{align*}
 As these integrals are absolutely convergent, we now apply Fubini's theorem to reverse the order of integration; then the inner integral with respect to $y$ is seen to be the Fourier transform of the Gaussian; on evaluating that integral we obtain  
 \begin{align}
  \label{Gammaequ}
 \phi_s(x) &= \frac{\sqrt{\pi}}{\Gamma((1+s)/2)} \int_{0}^{\infty}\xi^{(s/2)-1} e^{-\xi} e^{-x^2/4\xi}d\xi \nonumber \\
 &=\frac{\sqrt{\pi}}{\Gamma((1+s)/2)} \, |x/2|^{s/2}\int_{0}^{\infty} \xi^{(s/2)-1} \exp(-|x|(\xi+\xi^{-1})/2) d\xi.
 \end{align}

 In order to prove that $\Phi^{\circ}$ spans a dense subspace in $H_s^+$, it suffices to show that if $f \in H_s$ is even and is such that $\langle R(\gamma(a, 1))\phi_s |f\rangle_s= 0 $ for any $ a \in \R^{\times}$ then $f = 0$, almost everywhere.  It follows from the condition $\langle R(\gamma(a, 1))\phi_s |f\rangle = 0$ that 
 $$
 \int_{-\infty}^{\infty} |a|^{(1-s)/2}\phi_s(xa)\overline{f(x)}|x|^{-s}dx=0
 $$ 
 and that 
 $$
 \int_{-\infty}^{\infty}\phi_s(xa)\overline{f(x)}|x|^{-s}dx=0
 $$
 for all $a \in \R^{\times}$.  From (\ref{Gammaequ}), we have 
 $$
 \phi_s(xa) = \frac{\sqrt{\pi}}{\Gamma((1+s)/2)} |xa/2|^{s/2} \int_{0}^{\infty} \xi^{(s/2)-1} \exp(-|xa|(\xi+\xi^{-1})/2) d\xi,
 $$
 and replacing $\xi$ by $\xi/|x|$ in the latter integral, we obtain 
 $$
 \phi_s(xa) =\frac{\sqrt{\pi}}{\Gamma((1+s)/2)} |a|^{s} \int_{0}^{\infty} \xi^{\frac{s}{2}-1} \exp\Big(-|a|\xi-\frac{x^2}{4\xi}\Big) d\xi.
 $$
 Therefore, 
 $$
 \int_{-\infty}^{\infty}\overline{f(x)}|x|^{-s}\left[\int_{0}^{\infty} \xi^{\frac{s}{2}-1} \exp\Big(-|a|\xi-\frac{x^2}{4\xi}\Big)d\xi \right]dx=0
 $$
 for all $a \in \R^\times$.   Again applying Fubini's theorem to interchange the order of integration, we obtain 
  $$
  \int_{0}^{\infty} \xi^{(s/2)-1} \left[\int_{-\infty}^{\infty}\overline{f(x)}|x|^{-s} \exp(-x^2/4\xi) dx \right]e^{-|a| \xi}d\xi=0.
  $$
 As the latter integral is a Laplace transform, it follows that 
 $$
 \int_{-\infty}^{\infty}\overline{f(x)}|x|^{-s} \exp(-x^2/4\xi) dx = 0
 $$
  $\xi$-almost everywhere.  Since $f$ is even, it follows that
 $$
  \int_{0}^{\infty}\overline{f(x)}|x|^{-s} \exp(-x^2/4\xi) dx
  $$
  almost everywhere in $\xi > 0$; equivalently, 
  $$
   \int_{0}^{\infty}\overline{f(\sqrt{x})} |x|^{-(s+1)/2} \exp(-x/4\xi) dx = 0
   $$
 almost everywhere in $\xi > 0$.  Hence, the Laplace transform of $\overline{f(\sqrt{x})}|x|^{-(s+1)/2}$ is zero almost everywhere on $\R^{\times}$. Therefore
$\overline{f(\sqrt{x})}=0$ for almost every $x > 0$, which implies that $f ( x ) = 0$, \text{  a.e.  } $x > 0$. Because $f$ is even, we deduce that $f ( x ) = 0$ a.e. This  completes the proof of the Lemma.
 \end{proof}
 
 \begin{lem}
\label{psis}
 For $0 < \Re(s) < 1$, define
 \begin{equation}
  \psi_s(x)=\int_{-\infty}^{\infty} ye^{-ixy} (1+y^2)^{-(s+1)/2} dy
 \end{equation}
 Then $\psi_s \in  H_s$, and the set ${\mathbf{\Psi}}^{\circ} = \left\{ R(\gamma)\psi_s : \gamma \in D\right\}$ spans a dense subspace of $H_s^- = \left\{f \in H_s : f \text{ is odd} \right\}$.
 \end{lem}
 
 The proof of this result is similar to that of Lemma \ref{phi}.
 
 \medskip
 
 \noindent{\it Proof of Theorem \ref{main}}. By Lemma \ref{preserve} it follows that $\rho_s$ is unitary when restricted to $Q$.  Also, by Lemma \ref{Subgroup}, we need to prove that $\rho_s(p)$ is a unitary operator on $H_s$.  
 
 Denoting $\rho_s(p)$ by $W$, it is a straightforward calculation to verify the following properties for $W$:

\smallskip
\noindent
(i)  $W^2 = I$, the identity operator of $H_s$.

\smallskip
\noindent
(ii)  $WR(\gamma)= R(\gamma^{-1})W$ for $\gamma= \gamma(a, d) \in D$.

\smallskip
\noindent
(iii)  $W\phi_s =\phi_s$.

\smallskip
\noindent
It is also a simple calculation to verify that $\int_{-\infty}^{\infty}\phi_s(x)\overline{(R(\gamma)\phi_s)(x)}|x|^{-s}$ is real
 for any $\gamma \in D$. Hence by Proposition \ref{RandL} and Lemma \ref{phi}, $W$ is a generalized Watson transform of $H_s^+$ with respect to the unitary representations $R(\gamma)$ and $R(\gamma^{-1})$.  Therefore, $W$ is unitary on $H_s^+$. 
 
 Similarly, from Proposition \ref{RandL} and Lemma \ref{psis}, it follows that $W$ is unitary on $H_s^-$. Consequently, $W$
 is unitary on $H_s = H_s^+ \oplus H_s^-$. This completes the proof of the unitarity of the complementary series of $G$. 
 \qed

\section*{Acknowledgements}
I wish to express my deep gratitude to my research advisors, the late Professor Kenneth I. Gross (University of Vermont) and the late Professor Ray A. Kunze (University of Georgia), who introduced me to the subject of group representations and suggested that I study the generalized Watson transforms. I also thank Donald Richards (Penn State University) for his continual patience and encouragement. Their help and support have been invaluable.


\begin{thebibliography}{00}

 \bibitem{bargmann}
  V. Bargmann, Representations of the Lorentz group, {\it Ann. Math.}, {\bf 48} (1947), 568--640.

\bibitem{knapp}
 A. W. Knapp, {\sl  Representation Theory of Semisimple Groups: An Overview Based on Examples},  Princeton University Press, Princeton, NJ, 1986.
 
 \bibitem{pitt}
 H. R. Pitt,  A note on bilinear forms, {\it J. London Math. Soc.}, {\bf 11} (1936), 174--180. 
 
 \bibitem{stein}
 E. M. Stein, Interpolation of linear operators, {\it Trans. Amer. Math. Soc.}, {\bf 83} (1956), 482--492.

 \bibitem{zheng00}
  Q. Zheng, Generalized Watson transforms I: General theory, {\it Proc. Amer. Math. Soc.}, {\bf 128} (2000), 2777--2787.
   
 \bibitem{zheng_in_prep}
  Q. Zheng, Generalized Watson transforms III: Hankel transforms on symmetric cones, {\it in preparation}.
  
 \end{thebibliography}
\end{document}